\documentclass{article}[12pt]

\usepackage{amsmath,amssymb,amsthm}
\usepackage[affil-it]{authblk}
\usepackage{enumitem}

\makeindex

\newcommand{\Z}{\mathbb Z}
\newcommand{\1}{\mathbf 1}

\newtheorem{theorem}{Theorem}[section]
\newtheorem{lemma}[theorem]{Lemma}
\newtheorem{corollary}[theorem]{Corollary}

\theoremstyle{definition}
\newtheorem{definition}[theorem]{Definition}
\newtheorem{remark}[theorem]{Remark}
\newtheorem{observation}[theorem]{Observation}

\author{Krishnendu Paul, Shameek Paul%
\thanks{E-mail addresses: \texttt{krishnendupaul@ggdcgopi2.ac.in, shameek.paul@rkmvu.ac.in}}}
\affil{\small
Government General Degree College Gopiballavpur-II,
P.O. Beliaberah, Dist. Jhargram, 721517, India \\

\bigskip
Ramakrishna Mission Vivekananda Educational and Research Institute, P.O. Belur Math, Dist. Howrah, 711202, India}

\date{}
\title {$Q_p$-weighted zero-sum constants}

\begin{document}

\baselineskip=14.5pt

\maketitle

\begin{abstract}
A sequence $S=(x_1,\ldots, x_k)$ in $\Z_p$ is called a $(Q_p,\mathbf 1)$-weighted zero-sum sequence if there exist $a_1,\ldots,a_k\in Q_p$ such that $a_1x_1+\cdots+a_kx_k=0$ and $a_1+\cdots+a_k=0$. The constant $E_{Q_p,\mathbf 1}$ is defined to be the smallest positive integer $k$ such that every sequence of length $k$ in $\Z_p$ has a $(Q_p,\mathbf 1)$-weighted zero-sum subsequence of length $p$. We determine the constant $E_{Q_p,\mathbf 1}$ and the related constants $C_{Q_p,\mathbf 1}$ and $D_{Q_p,\mathbf 1}$. We also study some $(Q_p,B)$-weighted zero-sum constants where $B$ is a subset of $Q_p$.
\end{abstract}

\section{Introduction}\label{intro}

By a sequence $S$ in a set $X$ of length $k$, we mean an element of the set $X^k$. For every $n\in \mathbb N$ with $n\geq 2$ we denote the ring $\mathbb Z/n\mathbb Z$ by $\mathbb Z_n$. We let $A$ and $B$ be non-empty subsets of $\mathbb Z_n\setminus \{0\}$. A sequence $(x_1,\ldots, x_k)$ in $\mathbb Z_n$ is called an {\it $A$-weighted zero-sum sequence} if there exist $a_1,\ldots,a_k\in A$ such that $a_1x_1+\cdots+a_kx_k=0$. A sequence $(x_1,\ldots,x_k)$ in $\mathbb Z_n$ is called an {\it $(A,B)$-weighted zero-sum sequence} if there exist $a_1,\ldots,a_k\in A$ and $b_1,\ldots,b_k\in B$ such that $a_1x_1+\cdots+a_kx_k=0$ and $b_1a_1+\cdots +b_ka_k=0$.

For $a,b\in\mathbb Z$ we denote the set $\{x\in\mathbb Z:a\leq x\leq b\}$ by $[a,b]$. We let $|A|$ denote the number of elements in a finite set $A$. We denote the subsets $\{0\}$ and $\{1\}$ of $\mathbb Z_n$ by $\mathbf 0$ and $\mathbf 1$ respectively. A sequence which is a $\mathbf 1$-weighted zero-sum sequence is simply called a {\it zero-sum sequence}.

We let $U(n)$ denote the group of units of $\mathbb Z_n$. We reserve the letter $p$ to denote an odd prime. We let $Q_p=\big\{x^2:x\in U(p)\big\}$ and $N_p=U(p)\setminus Q_p$. Let $S=(x_1,\ldots,x_k)$ be a sequence in $\mathbb Z_n$ and let $x\in \mathbb Z_n$. We define the {\it translate} of $S$ by $x$ to be the sequence $S+x=(x_1+x,\ldots,x_k+x)$.

\begin{observation}\label{tr}
Suppose a sequence $S$ in $\mathbb Z_n$ is an $(A,\mathbf 1)$-weighted zero-sum sequence. Then every translate of $S$ is also an $(A,\mathbf 1)$-weighted zero-sum sequence, since the following identity holds.
\[a_1(x_1+x)+\cdots+a_k(x_k+x)=a_1x_1+\cdots +a_kx_k+(a_1+\cdots+a_k)x.\]
\end{observation}

\begin{remark}\label{trans}
If we want to show that a sequence $S$ in $\mathbb Z_n$ is an $(A,\1)$-weighted zero-sum sequence, then by Observation \ref{tr} we may assume that zero occurs the most number of times as a term of $S$.
\end{remark}

We now define the constants which we study in this article.

\begin{definition}
The constant $C_{A,B}(n)$ is the least positive integer $k$ such that any sequence in $\mathbb Z_n$ of length $k$ has an $(A,B)$-weighted zero-sum subsequence having consecutive terms.

The constant $D_{A,B}(n)$ is the least positive integer $k$ such that any sequence in $\mathbb Z_n$ of length $k$ has an $(A,B)$-weighted zero-sum subsequence.

The constant $E_{A,B}(n)$ is the least positive integer $k$ such that any sequence in $\mathbb Z_n$ of length $k$ has an $(A,B)$-weighted zero-sum subsequence of length $n$.
\end{definition}

The constants $C_A(n)$, $D_A(n)$, and $E_A(n)$ are defined to be the constants $C_{A,\mathbf 0}(n)$, $D_{A,\mathbf 0}(n)$, and $E_{A,\mathbf 0}(n)$ respectively. We see that $D_{A,B}(n)\leq C_{A,B}(n)$ and $D_{A,B}(n)\leq E_{A,B}(n)$. In \cite{KS} it has been shown that
\begin{equation}\label{ubzn}
C_{A,B}(n)\leq n^2 ~\mbox{ and }~ D_{A,B}(n)\leq E_{A,B}(n)\leq 2n-1.
\end{equation}

When the number $n$ is clear from the weight-set pair $(A,B)$, we denote the constants $C_{A,B}(n)$, $D_{A,B}(n)$, and $E_{A,B}(n)$ simply as $C_{A,B}$, $D_{A,B}$, and $E_{A,B}$. We use similar notation for the constants $C_A(n)$, $D_A(n)$, and $E_A(n)$.

\section{Determining the value of $E_{Q_p,\1}$}

\begin{remark}\label{eqp1}
In this section, we will show that
\[E_{Q_p,\1}=
\begin{cases}
p+2, & \text{if }p\neq 5; \\
9, & \text{if }p=5.
\end{cases}
\]
\end{remark}

\begin{theorem}\label{weak}
We have $E_{Q_p,\1}\in [p+2, p+4]$.
\end{theorem}

\begin{proof}
Since every $(Q_p,\1)$-weighted zero-sum sequence is a $Q_p$-weighted zero-sum sequence, we see that $E_{Q_p}\leq E_{Q_p,\1}$. From Theorem 3 of \cite{AR} we see that $E_{Q_p}=p+2$. Hence, we get that $E_{Q_p,\1}\geq p+2$.

Let $m=p+4$ and let $S=(a_1,\ldots,a_m)$ be a sequence in $\Z_p$. We show that $S$ has a $(Q_p,\1)$-weighted zero-sum subsequence of length $p$. Consider the following system of equations in the variables $X_1,\ldots,X_m$ over the field $\Z_p$.
\[a_1X_1^2\,+\,\cdots\,+\,a_mX_m^2=0~,~~~~ X_1^2\,+\,\cdots\,+\,X_m^2=0~,~~~X_1^{p-1}\,+\,\cdots\,+\,X_m^{p-1}=0.\]
By the Chevalley-Warning Theorem \cite{S}, this system has a non-trivial solution $(b_1,\ldots,b_m)$. Let $J=\big\{i\in [1,m]:b_i\neq 0\big\}$.

Since $U(p)$ is a group of order $p-1$ and $m=p+4$, from the third equation it follows that $|J|=p$. Let $T$ be the subsequence of $S$ such that $a_i$ is a term of $T$ if and only if $i\in J$. From the first two equations, we see that $T$ is a $(Q_p,\1)$-weighted zero-sum subsequence of $S$ having length $|J|$. Hence, it follows that $E_{Q_p,\1}\leq p+4$.
\end{proof}

\begin{lemma}\label{q51}
A sequence in $\Z_5$ having length five is a $(Q_5,\1)$-weighted zero-sum sequence if and only if it is a $(\1,\1)$-weighted zero-sum sequence.
\end{lemma}

\begin{proof}
Let $T=(x_1,\ldots,x_5)$ be a $(Q_5,\1)$-weighted zero-sum sequence in $\Z_5$. There exist $a_1,\ldots,a_5\in Q_5$ such that $a_1x_1+\cdots +a_5x_5=0$ and $a_1+\cdots +a_5=0$. Since $Q_5=\{1,-1\}$ and $a_1+\cdots +a_5=0$, we see that the sequence $(a_1,\ldots,a_5)$ is a constant sequence. It follows that $x_1+\cdots +x_5=0$. Hence, we see that $T$ is a $(\1,\1)$-weighted zero-sum sequence. Also, every $(\1,\1)$-weighted zero-sum sequence in $\Z_5$ is a $(Q_5,\1)$-weighted zero-sum sequence.
\end{proof}

\begin{corollary}\label{q5}
We have $E_{Q_5,\1}=9$.
\end{corollary}

\begin{proof}
By Lemma \ref{q51}, it follows that $E_{\1,\1}(5)=E_{Q_5,\1}$, and from \cite{KS} we see that $E_{\1,\1}(5)=9$.
\end{proof}

The next result is Lemma 1 of \cite{CM}.

\begin{lemma}\label{cm}
Let $p\geq 7$ and $x_1,x_2,x_3\in U(p)$. Then $x_1\,Q_p+x_2\,Q_p+x_3\,Q_p=\Z_p$.
\end{lemma}
 
\begin{corollary}\label{zs}
Let $p\geq 7$ and let $S$ be a sequence in $\Z_p$ which has at least three non-zero terms. Then $S$ is a $Q_p$-weighted zero-sum sequence.
\end{corollary}
 
\begin{proof}
Let $S=(x_1,\ldots,x_k)$ and let $x_1,x_2,x_3$ be non-zero. By Lemma \ref{cm} there exist $c_1,c_2,c_3\in Q_p$ such that $c_1x_1+c_2x_2+c_3x_3=-(x_4+\cdots+x_k)$.
\end{proof}
 
\begin{lemma}\label{nss}
Let $p\geq 7$ and let $T=(x_1,\ldots,x_k)$ be a sequence in $\Z_p$. If there exist $c_1,\ldots,c_k\in Q_p$ such that $c_1x_1+\cdots +c_kx_k=0$ and $c_1+\cdots +c_k\neq 0$, then $T+(0,0)$ is a $(Q_p,\1)$-weighted zero-sum sequence.
\end{lemma}
 
\begin{proof}
Let $c=c_1+\cdots +c_k$. Since $c\neq 0$, by Corollary \ref{zs} the sequence $(c,1,1)$ is a $Q_p$-weighted zero-sum sequence. Thus, there exist $d,d_2,d_3\in Q_p$ such that $d\,c+d_2+d_3=0$. Since $d\,(c_1x_1+\cdots +c_kx_k)+d_2\,0+d_3\,0=0$, it follows that $T+(0,0)$ is a $(Q_p,\1)$-weighted zero-sum sequence.
\end{proof}

\begin{lemma}\label{3z}
Let $p\geq 7$ and let $T$ be a sequence in $\Z_p$ which has at least three non-zero terms. Then $T+(0,0,0)$ is a $(Q_p,\1)$-weighted zero-sum sequence.
\end{lemma}

\begin{proof}
Let $T=(x_1,\ldots,x_k)$. By Corollary \ref{zs} there exist $c_1,\ldots,c_k\in Q_p$ such that $c_1x_1+\cdots+c_kx_k=0$. Let $c=c_1+\cdots+c_k$. Again, by Corollary \ref{zs} we see that the sequence $(c,1,1,1)$ is a $Q_p$-weighted zero-sum sequence. By a similar argument as in the proof of Lemma \ref{nss}, we see that $T+(0,0,0)$ is a $(Q_p,\1)$-weighted zero-sum sequence.
\end{proof}

The next result is a conjecture by Erd\H os and Heilbronn which was proved by Dias da Silva and Hamidoune (see \cite{Nat}).

\begin{theorem}\label{eh}
Let $A\subseteq \Z_p$ such that $|A|\geq 2$ and let
\[A\hat +A=\{a+b:a,b\in A\mbox{ and }a\neq b\}.\]
Then either $A\hat +A=\Z_p$ or $|A\hat + A|\geq 2\,|A|-3$.
\end{theorem}

\begin{lemma}\label{nspart}
Let $t\in \Z_p$ and let $z\in Q_p\hat + Q_p$. Suppose $x$ and $x'$ are distinct elements of $Q_p$. Then there exist $c,c'\in Q_p$ such that $cx+c'x'=z$ and $c+c'\neq t$.
\end{lemma}

\begin{proof}
Since $z\in Q_p\hat+Q_p$, there exist $u,u'\in Q_p$ such that $z=u+u'$ and $u\neq u'$. Let $a=u/x$, $a'=u'/x'$, $b=u'/x$, and $b'=u/x'$. We see that $a,a',b,b'\in Q_p$. Also, we have
\begin{equation}\label{qp}
ax+a'x'=z=bx+b'x'.
\end{equation}
So we get that
\[(a-b)x=(b'-a')x'.\]
We see that $a-b=(u-u')/x\neq 0$. So if $a-b=b'-a'$, we get the contradiction that $x=x'$. Hence, it follows that $a-b\neq b'-a'$ and so $a+a'\neq b+b'$. Thus, by \eqref{qp} there exist $c,c'\in Q_p$ such that $cx+c'x'=z$ and $c+c'\neq t$.
\end{proof}

\begin{lemma}\label{ns}
Let $p\geq 11$, let $t\in \Z_p$, and let $z\in U(p)$. Suppose $x$ and $x'$ are distinct elements which belong to the same coset of $Q_p$. Then there exist $c,c'\in Q_p$ such that $cx+c'x'=z$ and $c+c'\neq t$.
\end{lemma}

\begin{proof}
Suppose $x,x'\in Q_p$. Since $|Q_p|=(p-1)/2$, by Theorem \ref{eh} we see that $|Q_p\hat +Q_p|\geq p-4$. Since $p\geq 11$, we see that $|Q_p|\geq 5$.

So it follows that $Q_p\cap (Q_p\hat +Q_p)\neq \emptyset$ and hence there exists $y\in Q_p$ such that $y\in Q_p\hat +Q_p$. Since $y\in Q_p\hat +Q_p$, we observe that $y\,Q_p\subseteq Q_p\hat+Q_p$. Also, since $y\in Q_p$, we have $y\,Q_p=Q_p$. Thus, we see that $Q_p\subseteq Q_p\hat+Q_p$.

Since $|N_p|=|Q_p|$, we see that $N_p\cap (Q_p\hat +Q_p)\neq \emptyset$. So there exists $y'\in N_p$ such that $y'\in Q_p\hat +Q_p$. Since $y'\in Q_p\hat +Q_p$, we have that $y'\,Q_p\subseteq Q_p\hat+Q_p$. Also, since $y'\in N_p$, we have $y'\,Q_p=N_p$. Thus, we see that $N_p\subseteq Q_p\hat+Q_p$.

It follows that $Q_p\cup N_p\subseteq Q_p\hat +Q_p$ and hence $U(p)\subseteq Q_p\hat +Q_p$. Since $z\in U(p)$, it follows that $z\in Q_p\hat+Q_p$. By Lemma \ref{nspart} there exist $c,c'\in Q_p$ such that $cx+c'x'=z$ and $c+c'\neq t$.

Suppose $x,x'\in N_p$. Let $a\in N_p$. We observe that $ax,ax'\in Q_p$ and $az\neq 0$. From the previous case, there exist $c,c'\in Q_p$ such that $c(ax)+c'(ax')=az$ and $c+c'\neq t$. Since $cx+c'x'=z$, we are done.
\end{proof}

\begin{lemma}\label{ns7}
Let $p\geq 11$ and let $S=(x_1,\ldots,x_n)$ be a sequence in $\Z_p$ which has at least one non-zero term. Suppose $x$ and $x'$ are distinct elements which belong to the same coset of $Q_p$. Then there exist $c_1,\ldots,c_n,c,c'\in Q_p$ such that
\[c_1x_1+\cdots+c_nx_n+cx+c'x'=0\mbox{ and }c_1+\cdots+c_n+c+c'\neq 0.\]
\end{lemma}

\begin{proof}
Since there exists $i\in [1,n]$ such that $x_i\neq 0$, we see that
\[|x_1Q_p+\cdots +x_nQ_p|\geq |x_iQ_p|=|Q_p|\geq 5.\]
So there exists $z\in U(p)$ such that $z\in x_1Q_p+\cdots +x_nQ_p$. Hence, there exist $c_1,\ldots,c_n\in Q_p$ such that $z=c_1x_1+\cdots +c_nx_n$. Let $t=c_1+\cdots +c_n$. By Lemma \ref{ns} we see that there exist $c,c'\in Q_p$ such that $cx+c'x'=-z$ and $c+c'\neq -t$.
\end{proof}

\begin{lemma}\label{ns'}
Let $S=(x_1,\ldots,x_n)$ be a sequence in $\Z_7$ which has at least two non-zero terms. Suppose $x$ and $x'$ are distinct elements which belong to the same coset of $Q_7$. Then there exist $c_1,\ldots,c_n,c,c'\in Q_7$ such that
\[c_1x_1+\cdots+c_nx_n+cx+c'x'=0\mbox{ and }c_1+\cdots+c_n+c+c'\neq 0.\]
\end{lemma}

\begin{proof}
Let $x,x'\in Q_7$. Since $S$ has at least two non-zero terms and $|Q_7|=3$, by the Cauchy-Davenport theorem (\cite{Nat}), we see that $|x_1Q_7+\cdots +x_nQ_7|\geq 5$. Since $|x_1Q_7+\cdots +x_nQ_7|+|Q_7|>7$, we see that $(x_1Q_7+\cdots +x_nQ_7)\cap Q_7\neq \emptyset$. Thus, there exists $z\in Q_7$ such that $z\in x_1Q_7+\cdots +x_nQ_7$. Hence, there exist $c_1,\ldots,c_n\in Q_7$ such that $z=c_1x_1+\cdots +c_nx_n$.

We observe that $Q_7 \hat + Q_7=-Q_7$. Hence, we see that $-z\in Q_7\hat + Q_7$. Let $t=c_1+\cdots +c_n$. By Lemma \ref{nspart} there exist $c,c'\in Q_7$ such that $cx+c'x'=-z$ and $c+c'\neq -t$.

Suppose $x,x'\in N_7$. Let $a\in N_7$. We observe that $ax,ax'\in Q_7$ and that the sequence $(ax_1,\ldots,ax_n)$ has at least two non-zero terms. Thus, from the previous case, there exist $c_1,\ldots,c_n,c,c'\in Q_7$ such that
\[c_1(ax_1)+\cdots +c_n(ax_n)+c(ax)+c'(ax')=0\]
and $c_1+\cdots+c_n+c+c'\neq 0$. Since $c_1x_1+\cdots +c_nx_n+cx+c'x'=0$, we are done.
\end{proof}

\begin{theorem}\label{e'q}
We have $E_{Q_p,\1}=p+2$ when $p\neq 5$.
\end{theorem}

\begin{proof}
By Theorem \ref{weak} we see that $E_{Q_p,\1}\geq p+2$. Since $Q_3=\1$, by Theorem 2.2 of \cite{KS} we see that $E_{Q_3,\1}=E_{\1,\1}(3)=5$. So we may assume that $p\geq 7$.

Let $S$ be a sequence in $\Z_p$ of length $p+2$. We will show that $S$ has a $(Q_p,\1)$-weighted zero-sum subsequence of length $p$. It will then follow that $E_{Q_p,\1}= p+2$. Suppose $S$ has at most two non-zero terms. Then $S$ has at least $p$ zeroes and hence we get a $(Q_p,\1)$-weighted zero-sum subsequence of length $p$.

So we may assume that $S$ has at least three non-zero terms. Since $S$ has length $p+2$, it follows that $S$ has a term which is repeated. By Remark \ref{trans} we may assume that zero occurs the most number of times as a term of $S$. It follows that $S$ has at least two zeroes.

Suppose $S$ has at least three zeroes. Let $T$ be a subsequence of $S-(0,0,0)$ of length $p-3$ which has  at least three non-zero terms. By Lemma \ref{3z} we see that $T+(0,0,0)$ is a $(Q_p,\1)$-weighted zero-sum sequence. Since $T+(0,0,0)$ has length $p$, we are done.
 
So we may assume that $S$ has exactly two zeroes. Let $T$ be a subsequence of $S-(0,0)$ of length $p-2$. Then all the terms of $T$ are non-zero. Since $p\geq 7$, we see that $T$ has at least five terms. It follows that at least three terms of $T$ are in the same coset of $Q_p$. At least two of these three terms are distinct, since zero occurs the most number of times as a term of $S$.

Let $T=(x_1,\ldots,x_k)$ where $k=p-2$. By Lemmas \ref{ns7} and \ref{ns'} there exist $c_1,\ldots,c_k\in Q_p$ such that $c_1x_1+\cdots+c_kx_k=0$ and $c_1+\cdots+c_k\neq 0$. Hence, by Lemma \ref{nss} we see that $T+(0,0)$ is a $(Q_p,\1)$-weighted zero-sum sequence.
\end{proof}

\section{Determining the value of $C_{Q_p,\1}$}

\begin{lemma}\label{l2}
Let $p\geq 7$ and let $p\equiv 1~(\emph{mod}~4)$. Let $S$ be a sequence in $\Z_p$ which has at least two zeroes and at least three non-zero terms. Then $S$ is a $(Q_p,\1)$-weighted zero-sum sequence.
\end{lemma}

\begin{proof}
Let $T=S-(0,0)$. Suppose $T=(x_1,\ldots,x_k)$. By Corollary \ref{zs} there exist $c_1,\ldots,c_k\in Q_p$ such that $c_1x_1+\cdots+c_kx_k=0$. Let $c=c_1+\cdots+c_k$. If $c\neq 0$, by Lemma \ref{nss} we see that $T+(0,0)$ is a $(Q_p,\1)$-weighted zero-sum sequence.

Suppose $c=0$. We see that $T$ is a $(Q_p,\1)$-weighted zero-sum sequence. Since $p\equiv 1~(\text{mod}~4)$, we see that $-1\in Q_p$. Thus, we see that $(0,0)$ is a $(Q_p,\1)$-weighted zero-sum sequence. Hence, it follows that $T+(0,0)$ is a $(Q_p,\1)$-weighted zero-sum sequence.
\end{proof}

\begin{remark}
The statement of Lemma \ref{l2} is not valid when $p=5$, since $Q_5=\{1,-1\}$ and the sequence $(1,1,1,0,0)$ in $\Z_5$ is not a $\{1,-1\}$-weighted zero-sum sequence.
\end{remark}

\begin{corollary}\label{nsc}
Let $p\geq 11$. Suppose $x_1$, $x_2$, and $x_3$ are distinct non-zero elements in $\Z_p$. Then there exist $c_1,c_2,c_3\in Q_p$ such that $c_1x_1+c_2x_2+c_3x_3=0$ and $c_1+c_2+c_3\neq 0$.
\end{corollary}

\begin{proof}
We observe that given three distinct non-zero elements in $\Z_p$, at least two of them will belong to the same coset of $Q_p$. The result now follows from Lemma \ref{ns7}.
\end{proof}

\begin{remark}\label{cqp}
In this section, we will show that
\[C_{Q_p,\1}=
\begin{cases}
6, & \text{if }p\equiv 1~(\text{mod}~4); \\
9, & \text{if }p\equiv 3~(\text{mod}~4).
\end{cases}
\]
\end{remark}

\begin{observation}\label{qp1}
We see that every $(Q_p,\1)$-weighted zero-sum sequence has length at least two. Suppose there exists a $(Q_p,\1)$-weighted zero-sum sequence having length two. Then there exist $a,b\in Q_p$ such that $a+b=0$. Thus, we see that $-1\in Q_p$ which implies that $p\equiv 1~(\text{mod }4)$. Hence, if $p\equiv 3~(\text{mod }4)$, every $(Q_p,\1)$-weighted zero-sum sequence has length at least three.
\end{observation}

\begin{theorem}\label{lb}
We have $C_{Q_p,\1}\geq 6$. Further, $C_{Q_p,\1}\geq 9$ when $p\equiv 3~(\emph{mod}~4)$.
\end{theorem}

\begin{proof}
By Theorem 4 of \cite{SKS1} we have that  $C_{Q_p}=3$. Thus, there exists a sequence $S'=(x,y)$ in $\Z_p$ which does not have any $Q_p$-weighted zero-sum subsequence. Consider the sequence $S=(0,x,0,y,0)$. Suppose $T$ is a $(Q_p,\1)$-weighted zero-sum subsequence of consecutive terms of $S$. By Observation \ref{qp1} we see that $T$ has length at least two. This gives the contradiction that $S'$ has a $Q_p$-weighted zero-sum subsequence. Thus, we see that $S$ does not have any $(Q_p,\1)$-weighted zero-sum subsequence of consecutive terms. Hence, it follows that $C_{Q_p,\1}\geq 6$.
 
Let $p\equiv 3~(\text{mod}~4)$. Consider the sequence $S=(0,0,x,0,0,y,0,0)$. Suppose $T$ is a $(Q_p,\1)$-weighted zero-sum subsequence of consecutive terms of $S$. By Observation \ref{qp1} we see that $T$ has length at least three. This gives the contradiction that $S'$ has a $Q_p$-weighted zero-sum subsequence. Thus, we see that $S$ does not have any $(Q_p,\1)$-weighted zero-sum subsequence of consecutive terms. Hence, it follows that $C_{Q_p,\1}\geq 9$.
\end{proof}

\begin{lemma}\label{ls}
Let $p\equiv 1~(\emph{mod}~4)$. Suppose $x$ and $y$ are in the same coset of $Q_p$. Then $0\in x\,Q_p+y\,Q_p$.
\end{lemma}

\begin{proof}
Suppose $x,y\in Q_p$. Let $a=x^{-1}$ and $b=-y^{-1}$. Since $-1\in Q_p$, we see that $a,b\in Q_p$. Also, we have that $ax+by=0$.

Suppose $x,y\in N_p$. Let $c\in N_p$, let $x'=cx$, and $y'=cy$. We see that $x',y'\in Q_p$. By the previous case, there exist $a,b\in Q_p$ such that $ax'+by'=0$. Hence, we see that $ax+by=0$.
\end{proof}

\begin{theorem}\label{c1}
We have $C_{Q_p,\1}=6$ when $p\equiv 1~(\emph{mod}~4)$ and $p\neq 5$.
\end{theorem}

\begin{proof}
By Theorem \ref{lb} we see that $C_{Q_p,\1}\geq 6$. Let $S$ be a sequence in $\Z_p$ of length six. We want to show that $S$ has a $(Q_p,\1)$-weighted zero-sum subsequence of consecutive terms. By Remark \ref{trans} we may assume that zero occurs the most number of times as a term of $S$.

\noindent
\texttt{Case 1:} The sequence $S$ has at least two zeroes.

If at least three terms of $S$ are non-zero, then by Lemma \ref{l2} we see that $S$ is a $(Q_p,\1)$-weighted zero-sum sequence. So we may assume that at most two terms of $S$ are non-zero. Since $S$ has length six, it follows that $S$ has consecutive zeroes. Since $-1\in Q_p$, we see that $(0,0)$ is a $(Q_p,\1)$-weighted zero-sum subsequence of $S$.

\noindent
\texttt{Case 2:} The sequence $S$ has exactly one zero.

Since zero occurs the most number of times as a term of $S$, it follows that all the terms of $S$ are distinct. Since $S$ has five non-zero terms, we observe that at least two terms $y_1$ and $y_2$ of $S$ are in the same coset of $Q_p$. Let $x_1,x_2$, and $x_3$ be the other non-zero terms of $S$.

By Lemma \ref{ls} there exist $b_1,b_2\in Q_p$ such that $b_1y_1+b_2y_2=0$. As $y_1\neq y_2$, it follows that $b_1\neq -b_2$ and so $b_1+b_2\neq 0$. By Corollary \ref{nsc} there exist $a_1,a_2,a_3\in Q_p$ such that $a_1x_1+a_2x_2+a_3x_3=0$ and $a_1+a_2+a_3\neq 0$.

Let $a=a_1+a_2+a_3$ and $b=b_1+b_2$. Since $a$ and $b$ are non-zero, by Corollary \ref{zs} there exist $c_1,c_2,c_3\in Q_p$ such that $c_1a+c_2b+c_3=0$. Also, since $c_1(a_1x_1+a_2x_2+a_3x_3)+c_2(b_1y_1+b_2y_2)+c_30=0$, it follows that $S$ is a $(Q_p,\1)$-weighted zero-sum sequence.
\end{proof}

\begin{theorem}
We have $C_{Q_5,\1}=6$.
\end{theorem}

\begin{proof}
By Theorem \ref{lb} we see that $C_{Q_5,\1}\geq 6$. Let $S$ be a sequence in $\Z_5$ having length six. Then $S$ has a repeated term. By Remark \ref{trans} we may assume that zero occurs the most number of times as a term of $S$. It follows that $S$ has at least two zeroes. We want to show that $S$ has a $(Q_5,\1)$-weighted zero-sum subsequence of consecutive terms.

Suppose two consecutive terms of $S$ are equal. Since $-1\in Q_5$, we are done. So we may assume that no two consecutive terms of $S$ are equal. Since $S$ has no consecutive zeroes, it follows that $S$ has at least three non-zero terms.

\noindent
{\tt Case 1:} The sequence $S$ has at most three non-zero terms.

We see that either $S=(x_1,0,x_2,0,x_3,0)$ or $S=(0,x_1,0,x_2,0,x_3)$ where $x_1,x_2,x_3$ are non-zero terms. Let $S'=(x_1,x_2,x_3)$. By Theorem 4 of \cite{SKS1} we have $C_{Q_5}=3$. Hence, the sequence $S'$ has a $Q_5$-weighted zero-sum subsequence $T'$ of consecutive terms. We see that $T'$ has length at least two.

Suppose $T'=(x_1,x_2)$. There exist $a_1,a_2\in Q_5$ such that $a_1x_1+a_2x_2=0$. So we get that $a_1x_1-a_10+a_2x_2-a_20=0$ and hence $T=(x_1,0,x_2,0)$ is a $(Q_5,\1)$-weighted zero-sum subsequence of $S$. We can use a similar argument when $T'=(x_2,x_3)$ and when $T'=S'$.

\noindent
{\tt Case 2:} The sequence $S$ has at least four non-zero terms.

Since $S$ has at least two zeroes, we see that $S$ has exactly four non-zero terms, say $x_1,x_2,x_3,x_4$. By Lemma 2 of \cite{CM} there exist $c_i$'s in $Q_5$ such that $c_1x_1+\cdots +c_4x_4=0$.

If all the $c_i$'s are equal, then $x_1+x_2+x_3+x_4=0$. Let $T=S-(0,0)$. Since every term of $S$ is repeated at most twice, by considering the various possibilities for the $x_i$'s we see that $T$ is a permutation of either $(2,2,3,3)$, $(1,1,4,4)$, or $(1,2,3,4)$. For each of these sequences we can find $c_i$'s in $Q_5$ which are not all equal such that $c_1x_1+\cdots +c_4x_4=0$.

So we may assume that not all the $c_i$'s are  equal. Let $c=c_1+c_2+c_3+c_4$. Since $Q_5=\{1,-1\}$, we see that $c\notin Q_5$. Since $Q_5+Q_5=\mathbb Z_5\setminus Q_5$, there exist $d_1,d_2\in Q_5$ such that $c=d_1+d_2$. Since $c_1x_1+\cdots +c_4x_4-d_1 0-d_2 0=0$, we see that $S$ is a $(Q_5,\1)$-weighted zero-sum sequence.
\end{proof}

\begin{theorem}\label{c7}
We have $C_{Q_p,\1}=9$ when $p\equiv 3~(\emph{mod}~4)$ and $p\neq 7$.
\end{theorem}

\begin{proof}
By Theorem 2.4 of \cite{KS} we see that $C_{Q_3,\1}=9$, since $Q_3=\{1\}$. So we may assume that $p\geq 11$. By Theorem \ref{lb} we have that $C_{Q_p,\1}\geq 9$. Let $S$ be a sequence in $\Z_p$ having length nine. By Remark \ref{trans} we may assume that zero occurs the most number of times as a term of $S$.

If $S$ has at least three zeroes and at least three non-zero terms, by Lemma \ref{3z} we see that $S$ is a $(Q_p,\1)$-weighted zero-sum sequence. If at most two terms of $S$ are non-zero, then $T=(0,0,0)$ is a subsequence of consecutive terms of $S$. By Corollary \ref{zs} the sequence $(1,1,1)$ is a $Q_p$-weighted zero-sum sequence. So there exist $a,b,c\in Q_p$ such that $a+b+c=0$. Also, as $a0+b0+c0=0$, it follows that $T$ is a $(Q_p,\1)$-weighted zero-sum sequence.

So we may assume that $S$ has at most two zeroes. It follows that no term of $S$ is repeated more than twice.

\noindent
{\tt Case 1:} Either the sequence $S$ has exactly one zero, or both the first and last terms of $S$ are zeroes.

There exists a subsequence $T$ of consecutive terms of $S$ having length seven which has exactly one zero. Since no term of $S$ is repeated more than twice, there exist subsequences $T_1=(x_1,x_2,x_3)$ and $T_2=(y_1,y_2,y_3)$ of $T-(0)$ such that both $T_1$ and $T_2$ have no repeated terms and $T-(0)$ is a permutation of $T_1+T_2$. Since $p\geq 11$, by Corollary \ref{nsc} there exist $a_i$'s and $b_i$'s in $Q_p$ such that
\[a_1x_1+a_2x_2+a_3x_3=0,~b_1y_1+b_2y_2+b_3y_3=0,~a_1+a_2+a_3\neq 0,~b_1+b_2+b_3\neq 0.\]

Let $a=a_1+a_2+a_3$ and $b=b_1+b_2+b_3$. By Corollary \ref{zs} the sequence $(a,b,1)$ is a $Q_p$-weighted zero-sum sequence. So there exist $c_1,c_2,c_3\in Q_p$ such that $c_1a+c_2b+c_3=0$. Also, since
\[c_1(a_1x_1+a_2x_2+a_3x_3)+c_2(b_1y_1+b_2y_2+b_3y_3)+c_30=0,\]
it follows that $T$ is a $(Q_p,\1)$-weighted zero-sum sequence.
 
\noindent
{\tt Case 2:} The sequence $S$ has exactly two zeroes, and either the first term or the last term of $S$ is non-zero.

There exists a subsequence $T$ of consecutive terms of $S$ having length eight which has exactly two zeroes. Let $T_1$, $T_2$, $a$, and $b$ be defined as in the previous case. By Corollary \ref{zs} the sequence $(a,b,1,1)$ is a $Q_p$-weighted zero-sum sequence. Now, by a similar argument as in the previous case, it follows that $T$ is a $(Q_p,\1)$-weighted zero-sum sequence.
\end{proof}

\begin{theorem}
We have $C_{Q_7,\1}=9$.
\end{theorem}

\begin{proof}
By Theorem \ref{lb} we have that $C_{Q_7,\1}\geq 9$. Let $S$ be a sequence in $\Z_7$ having length nine. By Remark \ref{trans} we may assume that zero occurs the most number of times as a term of $S$. Thus, at least two terms of $S$ must be zero. If $S$ has at least three zeroes, we use a similar argument as in the second paragraph of the proof of Theorem \ref{c7}.

So we may assume that $S$ has exactly two zeroes. Let $T=S-(0,0)$. Since $|Q_7|=3$, there exists a subsequence $T_1=(x_1,x_2,x_3,x_4)$ of $T$ such that all the terms of $T_1$ are in the same coset of $Q_7$. Since every term of $S$ is repeated at most twice, there exist two terms of $T_1$ which are distinct. By Lemma \ref{ns'} we see that there exist $a_i$'s in $Q_7$ such that $a_1x_1+a_2x_2+a_3x_3+a_4x_4=0$ and $a_1+a_2+a_3+a_4\neq 0$.

Let $T-T_1=(y_1,y_2,y_3)$. By Corollary \ref{zs} there exist $b_i$'s in $Q_7$ such that $b_1y_1+b_2y_2+b_3y_3=0$. Let $a=a_1+a_2+a_3+a_4$ and $b=b_1+b_2+b_3$. By Corollary \ref{zs} the sequence $(a,b,1,1)$ is a $Q_7$-weighted zero-sum sequence. So there exist $c_i$'s in $Q_7$ such that $c_1a+c_2b+c_3+c_4=0$. Also, since
\[c_1(a_1x_1+a_2x_2+a_3x_3+a_4x_4)+c_2(b_1y_1+b_2y_2+b_3y_3)+c_30+c_40=0,\]
it follows that $S$ is a $(Q_7,\1)$-weighted zero-sum sequence.
\end{proof}

\section{Determining the value of $D_{Q_p,\1}$}

\begin{remark}\label{dqp}
In this section, we will show that
\[D_{Q_p,\1}=
\begin{cases}
4\text{ or }5, & \text{if }p\equiv 1~(\text{mod}~4); \\
5, & \text{if }p\equiv 3~(\text{mod}~4).
\end{cases}
\]
\end{remark}

\begin{theorem}\label{dqp2}
We have $D_{Q_p,\1}=4$ or $5$.
\end{theorem}

\begin{proof}
Let $S=(a_1,\ldots,a_5)$ be a sequence in $\Z_p$. We consider the following system of equations in the  variables $X_1,\ldots,X_5$ over the field $\Z_p$.
\[a_1X_1^2\,+\,\cdots\,+\,a_5X_5^2=0,\hspace{1cm}X_1^2\,+\,\cdots\,+X_5^2=0.\]
By the Chevalley-Warning Theorem \cite{S}, this system has a non-trivial solution $(b_1,\ldots,b_5)$. Let $J=\{i\in[1,5]:b_i\neq 0\}$. Let $T$ be the subsequence of $S$ such that $a_i$ is a term of $T$ if and only if $i\in J$. As $J\neq \emptyset$, it follows that $T$ is a $(Q_p,\1)$-weighted zero-sum sequence. Thus, we see that $D_{Q_p,\1}\leq 5$.

From \cite{AMP} we see that $D_{Q_p}=3$. So there exists a sequence $S'=(x,y)$ which has no $Q_p$-weighted zero-sum subsequence. Let $S=S'+(0)$. Suppose $T$ is a $(Q_p,\1)$-weighted zero-sum subsequence of $S$. By Observation \ref{qp1} we see that $T$ has length at least two. This gives the contradiction that $S'$ has a $Q_p$-weighted zero-sum subsequence. Thus, we see that $S$ does not have any $(Q_p,\1)$-weighted zero-sum subsequence. Hence, it follows that $D_{Q_p,\1}\geq 4$.
\end{proof} 
 
\begin{theorem}\label{dqp3}
We have $D_{Q_p,\1}=5$ when $p\equiv 3~(\emph{mod}~4)$.
\end{theorem}

\begin{proof} 
From \cite{AMP} we see that $D_{Q_p}=3$. So there exists a sequence $S'=(x,y)$ which has no $Q_p$-weighted zero-sum subsequence. Let $S=S'+(0,0)$. Suppose $T$ is a $(Q_p,\1)$-weighted zero-sum subsequence of $S$. By Observation \ref{qp1} we see that $T$ has length at least three. This gives the contradiction that $S'$ has a $Q_p$-weighted zero-sum subsequence. Thus, we see that $S$ does not have any $(Q_p,\1)$-weighted zero-sum subsequence. Hence, it follows that $D_{Q_p,\1}\geq 5$. Thus, from Theorem \ref{dqp2} we are done.
\end{proof}

\section{$(Q_p,B)$-weighted zero-sum constants}

\begin{theorem}\label{eqp}
Let $B\subseteq \Z_p$ where $p\neq 5$. We have $E_{Q_p,B}=p+2$.
\end{theorem}

\begin{proof}
Since every $(Q_p,\1)$-weighted zero-sum sequence is a $(Q_p,B)$-weighted zero-sum sequence, we see that $E_{Q_p,B}\leq E_{Q_p,\1}$.
Also, since every $(Q_p,B)$-weighted zero-sum sequence is a $Q_p$-weighted zero-sum sequence, we see that $E_{Q_p}\leq E_{Q_p,B}$. From Theorem 3 of \cite{AR} we see that $E_{Q_p}=p+2$. Now, by Theorem \ref{e'q}, we are done.
\end{proof}

The statement of Theorem \ref{eqp} does not hold when $p=5$ and $B=\1$, since from Corollary \ref{q5} we see that $E_{Q_5,\1}=9$. However, from the next result, we see that this statement holds when $B=Q_5$.

\begin{theorem}\label{eq5}
We have $E_{Q_5,Q_5}=7$.
\end{theorem}

\begin{proof}
Let $S$ be a sequence in $\Z_5$ having length $E_{Q_5}$. Then $S$ has a $Q_5$-weighted zero-sum subsequence $T=(x_1,\ldots,x_5)$. So there exist $a_1,\ldots,a_5\in Q_5$ such that $a_1x_1+\cdots +a_5x_5=0$. Since $Q_5=\{-1,1\}$ we see that $a_1^2+\cdots +a_5^2=0$. Thus, we see that $T$ is a $(Q_5,Q_5)$-weighted zero-sum sequence. It follows that $E_{Q_5,Q_5}\leq E_{Q_5}$. Also, since every $(Q_5,Q_5)$-weighted zero-sum sequence is a $Q_5$-weighted zero-sum sequence, we see that $E_{Q_5}\leq E_{Q_5,Q_5}$. From Theorem 3 of \cite{AR} we see that $E_{Q_5}=7$. Hence, we are done.
\end{proof}

\begin{observation}\label{qp2}
If $0\notin B$, we see that every $(Q_p,B)$-weighted zero-sum sequence has length at least two. Now suppose $B\subseteq Q_p$. If $T=(x_1,x_2)$ is a $(Q_p,B)$-weighted zero-sum sequence, there exist $a_1,a_2\in Q_p$ and $b_1,b_2\in B$ such that $a_1x_1+a_2x_2=0$ and $b_1a_1+b_2a_2=0$. Since $B\subseteq Q_p$, it follows that $-1\in Q_p$, which in turn implies that $p\equiv 1~(\text{mod }4)$. Hence, if $p\equiv 3~(\text{mod }4)$, any $(Q_p,B)$-weighted zero-sum sequence has length at least three.
\end{observation}

\begin{theorem}
Let $B\subseteq \Z_p$. We have
\[C_{Q_p,B}=
\begin{cases}
6, & \text{if }p\equiv 1~(\emph{mod}~4)\text{ and }0\notin B; \\
9, & \text{if }p\equiv 3~(\emph{mod}~4)\text{ and }B\subseteq Q_p.
\end{cases}
\]
\end{theorem}

\begin{proof}
Since every $(Q_p,\1)$-weighted zero-sum sequence is a $(Q_p,B)$-weighted zero-sum sequence, we see that $C_{Q_p,B}\leq C_{Q_p,\1}$. Now, by Remark \ref{cqp} and by replacing Observation \ref{qp1} and $(Q_p,\1)$ in the proof of Theorem \ref{lb} by Observation \ref{qp2} and $(Q_p,B)$ respectively, we are done.
\end{proof}

\begin{theorem}
Let $B\subseteq \Z_p$. We have
\[D_{Q_p,B}=
\begin{cases}
4\text{ or }5, & \text{if }p\equiv 1~(\emph{mod}~4)\text{ and }0\notin B; \\
5, & \text{if }p\equiv 3~(\emph{mod}~4)\text{ and }B\subseteq Q_p.
\end{cases}
\]
\end{theorem}

\begin{proof}
Since every $(Q_p,\1)$-weighted zero-sum sequence is a $(Q_p,B)$-weighted zero-sum sequence, we see that $D_{Q_p,B}\leq D_{Q_p,\1}$. Now, by Remark \ref{dqp} and by replacing Observation \ref{qp1} and $(Q_p,\1)$ in the second paragraph of the proof of Theorem \ref{dqp2} and in the proof of Theorem \ref{dqp3} by Observation \ref{qp2} and $(Q_p,\1)$ by $(Q_p,B)$, we are done.
\end{proof}

\section{Concluding remarks}

When $p\equiv 1~(\text{mod}~4)$, we believe that $D_{Q_p,\1}=4$. If this is true, it will follow that $D_{Q_p,B}=4$ when $p\equiv 1~(\text{mod}~4)$ and $0\notin B$. However, despite our best efforts, we have been unable to show this.

\end{document}